\documentclass[12pt,reqno]{amsart}

\usepackage{graphicx}
\usepackage[arrow,matrix,curve]{xy} 

\usepackage{amssymb, latexsym, amsmath, amscd, array, 
}

\newtheorem{theorem}{Theorem}[section]
\newtheorem{lemma}[theorem]{Lemma}
\newtheorem{proposition}[theorem]{Proposition}
\newtheorem{corollary}[theorem]{Corollary}
\newtheorem{conjecture}[theorem]{Conjecture}

\theoremstyle{definition}
\newtheorem{definition}[theorem]{Definition}

\numberwithin{equation}{section}
\numberwithin{figure}{section} 
\numberwithin{table}{section}

\DeclareMathOperator{\PD}{{\rm PD}}

\DeclareMathOperator{\disp}{{\rm disp}}

\DeclareMathOperator{\sys}{{\rm sys}}

\DeclareMathOperator{\area}{{\rm area}}

\newcommand{\Mob}{{\rm Mob}}

 \newcommand{\AAA}{{\mathcal{A}}}

\newcommand\Bra{{\rm Br}}

\newcommand\cf {\hbox{\it cf.\ }}

\newcommand\dist{{\rm dist}}

\newcommand{\length}{{\rm length}}

\newcommand \arc{\, {\rm arc}}

\newcommand \cqfd{\unskip\kern 6pt\penalty 500
\raise -2pt\hbox{\vrule\vbox to10pt{\hrule width 4pt
\vfill\hrule}\vrule}\par}                 

\def\adots{\mathinner{\mkern2mu\raise1pt\hbox{.}
\mkern3mu\raise4pt\hbox{.}\mkern1mu\raise7pt\hbox{.}}}
\def\hfl#1{\frac{\buildrel{#1}}{{\hbox to 12mm{\rightarrowfill}}}}
  
\newcommand\C{{\mathbb {C}}}

 \newcommand\R {{\mathbb R}}
\newcommand\RR {{\mathbb R}} \newcommand\RP {{\mathbb R}{\mathbb P}}
\newcommand\T {{\mathbb T}} \newcommand\Z {{\mathbb Z}}

\newcommand{\rp}{\mathbb R\mathbb P}

\long\def\forget#1\forgotten{} %

\long\def\forgett#1\forgottent{} %

\def\circ{\mathchoice%
 {\mathrel{\raise 1pt\hbox{$\scriptstyle\mathchar"020E$}}}
 {\mathrel{\raise 1pt\hbox{$\scriptstyle\mathchar"020E$}}}
 {\mathrel{\raise 1pt\hbox{$\scriptscriptstyle\mathchar"020E$}}}
 {}
}

\newcommand{\nc}{\newcommand} \nc{\on}{\operatorname}

\nc{\df}{\on{\it df}}

\nc{\conf}{\on{conf}}

\nc{\spt}{\on{spt}}
\nc{\norm}[1]{\| #1 \|}
\nc{\parallelleer}{\norm{\ }} 
\nc{\parallelh}{\norm h} 
\nc{\parallelk}{\norm k} 
\nc{\parallelx}{\norm x} 
\nc{\parallelhrr}{\norm {h_\RR}} 
\nc{\parallelom}{\norm \omega} 
\nc{\parallelomij}{\norm {\omega_{i_j}}} 
\nc{\parallelomx}{\norm {\omega_{x}}} 
\nc{\parallelpi}{\norm \pi} 
\nc{\parallelalf}{\norm \alpha} 
\nc{\parallelalfs}{\norm {\alpha_s}} 
\nc{\parallelalfi}{\norm {\alpha_i}} 
\nc{\parallelalfij}{\norm {\alpha_{i_j}}} 
\nc{\parallelbeta}{\norm \beta} 
\nc{\parallelbetat}{\norm {\beta_t}} 
\nc{\parallelhcapalf}{\norm {h \cap \alpha}} 
\nc{\parallelPDralf}{\norm {\PD_\RR(\alpha)}} 
\nc{\strichleer}{| \  |}
\nc{\NN}{\mathbb N}

\nc{\rr}{\mbox{$\scriptstyle\mathbb R$}}

\nc{\dF}{{\it dF}} 

\nc{\DF}{{\it DF}} 

\nc{\ds}{{\it ds}} 

\nc{\dvol}{{\it dvol}}

\nc{\grad}{{\rm grad}} 
\nc{\strichw}{\|\omega\|} 
\nc{\strichwx}{|\omega_x|}
\nc{\Hess}{{\rm Hess}}

\begin{document}

\title{Hyperellipticity and systoles of Klein surfaces}

\author[M.~Katz]{Mikhail G. Katz$^{*}$}

\thanks{$^{*}$Supported by the Israel Science Foundation
grant~1294/06}

\address{Department of Mathematics, Bar Ilan University, Ramat Gan
52900 Israel} 

\email{katzmik@macs.biu.ac.il}

\author[S.~Sabourau]{St\'ephane Sabourau}


\address{Laboratoire d'Analyse et Math\'ematiques Appliqu\'ees, UMR 8050,
Universit\'e Paris-Est,
61 Avenue du G\'en\'eral de Gaulle, 94010 Cr\'eteil, France}

\email{sabourau@lmpt.univ-tours.fr}






\date{\today}

\begin{abstract}
Given a hyperelliptic Klein surface, we construct companion Klein
bottles, extending our technique of companion tori already exploited
by the authors in the genus~$2$ case.  Bavard's short loops on such
companion surfaces are studied in relation to the original surface so
to improve a systolic inequality of Gromov's.  A basic idea is to use
length bounds for loops on a companion Klein bottle, and then analyze
how curves transplant to the original non-orientable surface.  We
exploit the real structure on the orientable double cover by applying
the coarea inequality to the distance function from the real locus.
Of particular interest is the case of Dyck's surface.  We also exploit
an optimal systolic bound for the M\"obius band, due to Blatter.
\end{abstract}

\maketitle

\tableofcontents

\section{Introduction}

Systolic inequalities for surfaces compare length and area, and can
therefore be thought of as ``opposite'' isoperimetric inequalities.
The study of such inequalities was initiated by C. Loewner in 1949
when he proved his torus inequality for~$\T^2$ (see Pu \cite{Pu} and
Horowitz {\em et al.\/}~\cite{HKK}).  The systole, denoted ``sys'', of
a space is the least length of a loop which cannot be contracted to a
point in the space, and is therefore a natural generalisation of the
{\em girth\/} invariant of graphs.

\begin{theorem}[Loewner's torus inequality]
Every metric on the~$2$-dimensional torus~$\T$ satisfies the bound
\begin{equation}
\label{11e}
\sys(\T) \leq
2^{\tfrac{1}{2}}_{\phantom{I}}3^{-\tfrac{1}{4}}_{\phantom{I}}
\sqrt{\area(\T)}.
\end{equation}
\end{theorem}

In higher dimensions, M.~Gromov's deep result \cite{Gr1}, relying on
filling invariants, exhibits a universal upper bound for the systole
in terms of the total volume of an essential manifold.  L.~Guth
\cite{Gu11} recently found an alternative proof not relying on filling
invariants, and giving a generalisation of Gromov's inequality, see
also Ambrosio and Katz~\cite{AK}.

In dimension~$2$, the focus has been, on the one hand, on obtaining
near-optimal asymptotic results in terms of the genus \cite{KS2, KSV},
and on the other, on obtaining realistic estimates in cases of low
genus \cite{KS1, HKK}.  One goal has been to determine whether all
aspherical surfaces satisfy Loewner's bound \eqref{11e}, a question
that is still open in general.  We resolved it in the affirmative for
genus~$2$ in \cite{KS1}.  An optimal inequality of
C.~Bavard~\cite{Bav1} for the Klein bottle~$K$ is stronger than
Loewner's bound:
\begin{equation}
\label{11f}
\sys(K) \leq C_{{\rm Bavard}} \sqrt{\area(K)}, \quad\quad C_{{\rm
Bavard}}=
\pi^{\tfrac{1}{2}}_{\phantom{I}}8^{-\tfrac{1}{4}}_{\phantom{I}}
\approx 1.0539.
\end{equation}
Gromov proved a general estimate for all aspherical surfaces:
\begin{equation}
\label{11b}
\sys^2 \leq \frac{4}{3} \area,
\end{equation}
see \cite[Corollary~5.2.B]{Gr1}.  As Gromov points out in \cite{Gr1},
the~$\frac{4}{3}$ bound~\eqref{11b} is actually {\em optimal\/} in the
class of Finsler metrics.  Therefore any further improvement is not
likely to result from a mere application of the coarea formula.  One
can legitimately ask whether any improvement is in fact possible, of
course in the framework of Riemannian metrics.

The goal of the present article is to furnish such an improvement
in the case of non-orientable surfaces.  We will say that such a
surface is hyperelliptic if its orientable double cover is.

\begin{theorem} \label{theo:main}
Let~$n\geq 2$.  Every Riemannian metric from a hyperelliptic conformal
type on the surface~$n\rp^2$ satisfies the bound
\begin{equation}
\label{12b}
\sys^2 \leq 1.333 \; \area.
\end{equation}
\end{theorem}

Note the absence of an ellipsis following ``333'', making our estimate
an improvement on Gromov's~$\frac{4}{3}$ bound \eqref{11b}.  By
keeping track of the best constants in our estimates throughout the
proof of the theorem, one could obtain a slightly better bound.
However, our goal is merely to develop techniques sufficient to
improve the $\frac{4}{3}$ bound.  Since every conformal class on
Dyck's surface~$ K\# \rp^2 = \T^2 \# \rp^2 = 3\rp^2$ of Euler
characteristic~$-1$ is hyperelliptic, we have

\begin{corollary}
Every Riemannian metric on Dyck's surface~$3\rp^2$ satisfies
\[
\sys(3\rp^2)^2 \leq 1.333 \, \area(3\rp^2).
\]
\end{corollary}

\forget
In \cite{KS1}, the following result was proved.

\begin{proposition}
\label{s>3}
Every metric in the conformal class of a hyperelliptic
surface~$\Sigma_g$ of genus~$g$ satisfies the
estimate~$\frac{\sys^2}{\area} \leq \frac{4}{g+1}$.
\end{proposition}

Given a non-orientable surface with hyperelliptic double cover, we
apply Proposition~\ref{s>3} to conclude that the bound~\eqref{12b} is
satisfied whenever~$n\geq 7$.  It remains therefore to handle the
case~$n=3$.  
\forgotten

The proof of the main theorem exploits a variety of techniques ranging
from hyperellipticity to the coarea formula and cutting and pasting.
The current best upper bound for the systole on Dyck's surface only
differs by about~$30\%$ from the best known example given by a
suitable extremal hyperbolic surface (see Section~\ref{fourteen}).

We will exploit the following characterisation of the systole of a
non-orientable surface.  Given a metric on a Klein surface~$X=\tilde
X/\tau$ where~$\tau: \tilde X \to \tilde X$ is fixed point-free, we
consider the natural~$\tau$-invariant pullback metric on its
orientable double cover~$\tilde X$.

\begin{definition} 
The least displacement ``$\disp$'' of~$\tau:\Sigma_g\to\Sigma_g$ is
the number
\begin{equation}
\label{61b}
\disp(\tau)= \min \left\{ \dist(x,\tau(x)) \mid x\in \Sigma_g
\right\}.
\end{equation}
\end{definition}

\begin{proposition}
The systole of a Klein surface~$X$ can be expressed as the least of
the following two quantities:
\[
\sys(X) = \min \left\{ \sys(\tilde X), \disp(\tau) \right\}.
\]
\end{proposition}
 
Indeed, lifting a systolic loop of~$X$ to~$\tilde X$, we obtain either
a loop in the orientable cover~$\Sigma_g$, or a path connecting two
points which form an orbit of the deck transformation~$\tau$.

Recall that the following four properties of a closed surface~$\Sigma$
are equivalent: (1)~$\Sigma$ is aspherical; (2) the fundamental group
of~$\Sigma$ is infinite; (3) the Euler characteristic of~$\Sigma$ is
non-positive; (4)~$\Sigma$ is not homeomorphic to either~$S^2$
or~$\RP^2$.  The following conjecture has been discussed in the
systolic literature, see~\cite{SGT}.

\begin{conjecture}
Every aspherical surface satisfies Loewner's bound
\begin{equation}
\label{loew}
\frac{\sys^2}{\area} \leq \frac{2}{\sqrt{3}}.
\end{equation}
\end{conjecture}
M. Gromov \cite{Gr1} proved an asymptotic estimate which implies that
every orientable surface of genus greater than~$50$ satisfies
Loewner's bound.  This was extended to orientable surfaces of genus at
least~$20$ in \cite{KS2}, and for the genus~$2$ surface in \cite{KS1}.

Recent publications in systolic geometry include Ambrosio \&
Katz~\cite{AK}, Babenko \& Balacheff \cite{BB10}, Balacheff {\em et
al.\/}~\cite{BPS}, Belolipetsky~\cite{Bel11}, El~Mir~\cite{Elm10},
Fetaya~\cite{Fe}, Katz {\em et al.\/}~\cite{KK2, KW, KSV2}, Makover \&
McGowan \cite{MM}, Parlier~\cite{Par10}, Ryu~\cite{Ry}, and
Sabourau~\cite{Sa11}.

In Section~\ref{six}, we will review the relevant conformal
information, including hyperellipticity.  In Section~\ref{two}, we
will present metric information in the context of a configuration of
the five surfaces appearing in our main argument.
Section~\ref{fourteen} exploits optimal inequalities of Blatter and
Sakai for the M\"obius band so to prove our first theorem for $n=3$.
We handle the remaining case, namely $n \geq 4$, in
Section~\ref{sec:last}.

\section{Review of conformal information and hyperellipticity}
\label{six}

In this section we review the necessary pre-metric (i.e., conformal)
information.  The quadratic equation~$y^2=p$ over~$\C$ is well known
to possess two distinct solutions for every~$p\not=0$, and a unique
solution for~$p=0$.  Now consider the locus (solution set) of the
equation
\begin{equation}
\label{11}
y^2= p(x)
\end{equation}
for~$(x,y)\in \C^2$ and generic~$p(x)$ of even degree~$2g+2$.  Such a
locus defines a Riemann surface which is a branched two-sheeted cover
of~$\C$.  The cover is obtained by projection to the~$x$-coordinate.
The branching locus corresponds to the roots of~$p(x)$. 

There exists a unique smooth closed Riemann surface~$\Sigma_g$
naturally associated with \eqref{11}, sometimes called the {\em smooth
completion\/} of the affine surface~\eqref{11}, together with
a holomorphic map
\begin{equation}
\label{12}
Q_g: \Sigma_g\to \hat \C=S^2
\end{equation}
extending the projection to the~$x$-coordinate.  By the
Riemann-Hurwitz formula, the genus of the smooth completion is~$g$,
where~$\deg(p(x))=2g-2$.  All such surfaces are hyperelliptic by
construction.  The hyperelliptic involution~$J: \Sigma_g \to \Sigma_g$
flips the two sheets of the double cover of~$S^2$ and has
exactly~$2g+2$ fixed points, called the Weierstrass points of
$\Sigma_g$.  The hyperelliptic involution is unique.  The
involution~$J$ can be identified with the nontrivial element in the
center of the (finite) automorphism group of $X$ (\cf
\cite[p.~108]{FK}) when it exists, and then such a~$J$ is unique
\cite[p.204]{Mi}.

A hyperelliptic closed Riemann surface~$\Sigma_g$ admitting an
orientation-reversing (antiholomorphic) involution~$\tau$ can always
be reduced to the form \eqref{11} where~$p(x)$ is a polynomial all of
whose coefficients are real, where the involution~$\tau: \Sigma_g \to
\Sigma_g$ restricts to complex conjugation on the affine part of the
surface in~$\C^2$, namely
\begin{equation}
\tau(x,y)=(\bar x, \bar y).
\end{equation}
The special case of a fixed point-free involution~$\tau$ can be
represented as the locus of the equation
\begin{equation}
\label{21}
-y^2 = \prod_j (x-x_j)(x-\bar x_j),
\end{equation}
where~$x_j\in \C\setminus \R$ for all~$j$.  Here the minus sign on the
left hand side ensures the absence of real solutions, and therefore
the fixed point-freedom of~$\tau$.  The uniqueness of the
hyperelliptic involution implies the following.

\begin{proposition}
\label{comm}
We have the commutation relation~$\tau \circ J = J \circ \tau$.
\end{proposition}

A Klein surface~$X$ is a non-orientable closed surface.  Such a
surface can be thought of as an antipodal quotient~$X=\Sigma_g/\tau$
of an orientable surface by a fixed point-free, orientation-reversing
involution~$\tau$.  The pair~$(\Sigma_g,\tau)$ is known as a real
Riemann surface.  A Klein surface is homeomorphic to the connected sum
\[
X=n\RP^2 =\underbrace{\RP^2\#\ldots\#\RP^2}_n,
\]
of~$n$ copies of the real projective plane.  The case~$n=2$
corresponds to the Klein bottle~$K=2\RP^2$, covered by the torus.  In
the case~$n=3$, we obtain the surface~$3\RP^2$ of Euler
characteristic~$-1$, whose orientable double cover is the genus~$2$
surface~$\Sigma_2$.

In the sequel, we will focus on Dyck's surface $3\rp^2$ since the
proof of the main theorem in this case requires special arguments.

We will denote by~$\tilde X$ the orientable double cover of a Klein
surface~$X$.  If a surface~$X$ is hyperelliptic, we will denote by
$\Bra(X) \subset \hat \C$ its branch locus, i.e., the set of isolated branch points
of the double cover~$X\to \hat \C$.

\begin{definition}
A Klein bottle~$K$ is called a {\em companion\/} of a Klein surface
$X$ if we have the inclusion of the branch loci at the level of
the orientable double covers:
\[
\Bra(\tilde K) \subset \Bra(\tilde X)\subset \hat \C.
\]
\end{definition}

\begin{proposition}
\label{com}
Each Klein surface~$K$ homeomorphic to~$3\RP^3$ admits a triplet of
companion Klein bottles~$K_1, K_2, K_3$ satisfying the following three
conditions:
\begin{enumerate}
\item
$|\Bra(\tilde K_j)|=4$;
\item
$|\Bra(\tilde K_i) \cap \Bra(\tilde K_j)|=2$ for~$i\not=j$;
\item
$\cup_{j=1}^3 \Bra(\tilde K_j) = \Bra (\tilde X)$.
\end{enumerate}
\end{proposition}

\begin{proof}
Given a real Riemann surface~$(\Sigma_g, \tau)$, consider the
presentation~\eqref{21} with~$p(x)$ a real polynomial.  We can write
the roots of~$p$ as a collection of conjugate pairs~$(a, \bar a)$.
Thus in the genus~$2$ case, the affine form of the surface is the
locus of the equation
\begin{equation}
\label{51}
-y^2 = (x-a)(x-\bar a) (x-b)(x-\bar b) (x-c)(x-\bar c)
\end{equation}
in~$\C^2$.  Choosing two conjugate pairs, for instance~$(a, \bar a, b,
\bar b)$, we can construct a companion surface
\begin{equation}
\label{52}
-y^2 = (x-a)(x-\bar a) (x-b)(x-\bar b),
\end{equation}
By the Riemann-Hurwitz formula, its genus is one, and therefore the
(smooth completion of the) companion surface is a torus.  We will
denote it~$\T_{a,b}$.  By construction, its set of zeros
is~$\tau$-invariant.  In other words, the (affine part in~$\C^2$ of
the) torus is invariant under the action of complex conjugation.
Thus, the surface~$\T_{a,b}/\tau$ is a Klein bottle~$K$, namely, a
companion Klein bottle of the original Klein
surface~$3\RP^2=\Sigma_2/\tau$ stemming from~$\eqref{51}$.  We thus
obtain the three Klein bottles~$\T_{a,b}/\tau$,~$\T_{b,c}/\tau$, and
$\T_{c,a}/\tau$, proving the proposition.
\end{proof}

The maps constructed so far can be represented by the following
diagram of homomorphisms (note that two out of the four arrows point
in the leftward direction):
\begin{equation}
\label{53}
3\RP^2 \leftarrow \Sigma_2 \to S^2 \leftarrow \T_{a,b} \to K .
\end{equation}
Complex conjugation~$\tau$ on~$\hat \C = S^2$ fixes a circle called
the equatorial circle, denoted~$\hat\R\subset\hat\C$.

\begin{definition}
The {\em equator\/} of~$\Sigma_2$ is the circle given by the inverse
image of the equator~$\hat \R$ under the hyperelliptic
projection~$\Sigma_2 \to S^2$.
\end{definition}

\begin{lemma}
The equator of~$\Sigma_2$ coincides with the fixed point set of the
composition~$\tau \circ J$.  The equator is invariant under the action
of~$\tau$.  The action of~$\tau$ on the equator of~$\Sigma_2$ is fixed
point-free.
\end{lemma}

\begin{proof}
The lemma is immediate from Proposition~\ref{comm}.
\end{proof}

Similarly, we obtain the following.

\begin{lemma} \label{lem:circle}
Relative to the double cover~$\T_{a,b}\to\hat\C$, the inverse image of
the equatorial circle~$\hat\R\subset \hat \C$ is a pair of disjoint
circles, the involution~$\tau$ acts on the torus by switching the two
circles, while the involution~$\tau \circ J$ fixes both circles
pointwise.
\end{lemma}

A loop on a non-orientable surface is called~$2$-sided if its tubular
neighborhood is homeomorphic to an annulus, and~$1$-sided if its
tubular neighborhood is homeomorphic to a M\"obius band.  The following
lemma is immediate from the homotopy lifting property.

\begin{lemma}
\label{ed}
Let~$X$ be a Klein surface, and~$\tilde X$ its orientable double
cover.  Let~$\gamma\subset X$ be a one-sided loop, and~$\delta \subset
X$ a~$2$-sided loop in~$X=\tilde X/\tau$.  We have the following four
properties:
\begin{enumerate}
\item
Lifting~$\gamma$ to~$\tilde X$ yields a path on~$\tilde X$ connecting
a pair of points which form an orbit of the involution~$\tau$;
\item
lifting~$\delta$ to~$\tilde X$ yields a closed curve on~$X$;
\item
the inverse image of~$\gamma$ under the double cover~$\tilde X \to X$
is a circle (i.e. has a single connected component homeomorphic to a
circle);
\item
the inverse image of~$\delta$ under the double cover~$\tilde X \to X$
has a pair of connected components (circles).
\end{enumerate}
\end{lemma}

The following blend of topological and hyperelliptic information will
be helpful in the sequel.  The upperhalf plane in~$\C$ is a
fundamental domain for the action of complex conjugation~$\tau$.
Hence points on the Klein surface~$3\RP^2= \Sigma_2/\tau$ can be
represented by points in the closure of the upperhalf plane.  Consider
the northern hemisphere
\[
\hat\C^+\subset \hat\C=S^2,
\]
with the equator included.  We will think of the surface~$3\RP^2$ as a
double cover
\begin{equation}
\label{dc}
3\RP^2\to \hat\C^+ .
\end{equation}
The double cover~\eqref{dc} is branched along the equator as well as
at three additional Weierstrass points, corresponding to the
points~$a,b,c$ of the usual hyperelliptic~cover~$\Sigma_2\to\hat\C$.
We also have an analogue of the hyperelliptic involution, namely the
deck transformation
\begin{equation}
\label{deck}
J: 3\RP^2 \to 3\RP^2,
\end{equation}
fixing the equator and the three Weierstrass points.  Note that the three
remaining Weierstrass points~$\bar a, \bar b, \bar c\in \Sigma_2$ are
mapped to~$a,b,c$ by the involution~$\tau$.

\begin{lemma}
A simple loop~$\Delta\subset \hat\C^+$ parallel to the equator
decomposes the northern hemisphere into a union of a disk~$D$
``north'' of~$\Delta$ and an annulus~$A$ ``south'' of~$\Delta$:
\begin{equation}
\label{ann} 
\hat\C^+= D\cup_{\Delta} A,
\end{equation}
where~$a,b,c\in D$ (all the branch points of~$\hat\C^+$ are in the
disk~$D$).  The corresponding decomposition of~$3\RP^2$ is
\begin{equation}
\label{28}
3\RP^2 = \Sigma_{1,1} \cup_{S^1}^{\phantom{I}} \Mob,
\end{equation}
where~$\Sigma_{1,1}$ is the once-holed torus, and~$\Mob$, the M\"obius
band.
\end{lemma}

We will refer to this decomposition as the annulus decomposition,
since it is the non-orientable analogue of the
decomposition~\eqref{ann}, see Figure~\ref{hemisphere}.
In terms of the connected sum notation, the decomposition~\eqref{28}
corresponds to the topological decomposition~$3\RP^2 = \T^2 \# \RP^2$.

\begin{figure}[htbp]
\setlength\unitlength{1pt}
\begin{picture}(0,0)(0,0)
\put(140,-120){{\Large $\Delta$}}

\put(0,-60){$\bullet$} \put(8,-55){{\Large $c$}}

\put(-20,-75){$\bullet$} \put(-32,-80){{\Large $a$}}

\put(30,-70){$\bullet$} \put(38,-75){{\Large $b$}}

\end{picture}

\begin{center}
\includegraphics[height=90mm]{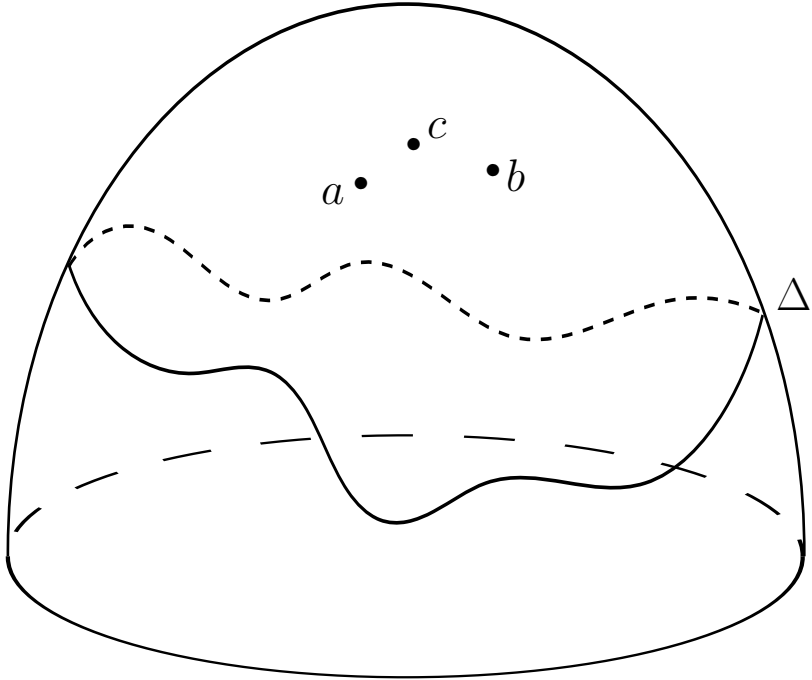}
\end{center}
\caption{Klein surface as a hemispherical double cover}
\label{hemisphere}
\end{figure}

The main idea in the proof of the main theorem is to reduce the
general situation to a case of a {\em metrically controlled\/} annulus
decomposition of the Klein surface, as explained in
Proposition~\ref{141} below.

\section{Reduction to the annulus decomposition}
\label{two}

A noncontractible loop on~$3\RP^2$ satisfying Bavard's
bound~\eqref{11f} will be called a Bavard loop.

\begin{proposition}
\label{141}
The Klein surface~$X=3\RP^2$ either contains a Bavard loop, or
admits a simple loop~$\delta\subset X$ which separates it as follows:
\[
X=\Sigma_{1,1}\cup_{\delta}^{\phantom{I}}\Mob,
\]
such that moreover
\begin{enumerate}
\item
the surface~$\Sigma_{1,1}$ contains the three Weierstrass points;
\item
the M\"obius band~$\Mob$ contains the equatorial circle;
\item
the loop~$\delta\subset X$ is~$J$-invariant;
\item
the loop~$\delta$ is of length at most~$2C_{{\rm
Bavard}}\sqrt{\area}$;
\item
the loop~$\delta$ double-covers a loop~$\Delta\subset\hat\C^+$ which
lifts to a systolic loop on a companion torus.
\end{enumerate}
\end{proposition}

The proposition will be proved in this section.  We start with some
preliminary observations.  Pulling back the metric to a double cover
results in a metric of twice the area.  We therefore obtain the
following lemma.

\begin{lemma}
Given a~$J$-invariant metric on the Klein surface~$3\RP^2$, we have
the following relations among the areas of the surfaces appearing in
diagram~\eqref{53}:
\[
\area(3\RP^2) = \area (S^2) = \area (K),
\]
as well as the relation~$\area(\Sigma_2) = \area (\T_{a,b}) = 2 \area
(3\RP^2)$.
\end{lemma}

Next, we show that averaging the metric improves the systolic ratio.

\begin{proposition}
\label{32}
Let~$\Sigma$ be a hyperelliptic orientable double cover of a Klein
surface~$X$.  Consider a Riemannian metric on~$\Sigma$.  Then
\begin{enumerate}
\item
averaging its metric by the hyperelliptic involution~$J$ increases the
systolic ratio~$\frac{\sys^2}{\area}$ of~$\Sigma$;
\item
averaging its metric by the complex conjugation~$\tau$ increases the
systolic ratio~$\frac{\sys^2}{\area}$ of~$\Sigma$;
\item
the ratio~$\frac{\disp(\tau)^2}{\area}$ defined in \eqref{61b}
increases under averaging.
\end{enumerate}
\end{proposition}

\begin{proof}
This point was discussed in detail in \cite{BCIK1}.  We summarize the
argument as follows.  Express the metric in terms of a constant
curvature metric in its conformal class, by means of a conformal
factor~$f^2$.  Thus, in the case of a torus we obtain a metric~$f^2(p)
(dx^2+dy^2)$ at a typical point~$p$ of the torus, where the
function~$f$ is doubly periodic.  We average the factor~$f^2$ by the
hyperelliptic involution~$J:\Sigma \to \Sigma$, i.e., we replace
$f^2(p)$ by
\[
\frac{1}{2}\left( f^2(p)+f^2(J(p)) \right).
\]
Such averaging preserves the total area of the metric.  Similarly, it
preserves the energy of a curve on the surface.  Choosing a constant
speed parametrisation of a systolic loop for the original metric, we
see that its energy is preserved under averaging.  Hence its length is
not decreased by averaging.  Similar remarks apply in the two
remaining cases.
\end{proof}

Given a Klein surface~$X=3\RP^2$, we can similarly average the metric
by the hyperelliptic involution~\eqref{deck}.  Hence we may assume
without loss of generality that the metrics on both~$3\RP^2$ and
$\Sigma_2$ are~$J$-invariant.

We now consider a companion Klein bottle~$K$ as in
Proposition~\ref{com}.  We will seek to transplant short loops
from~$K$ to~$X$.  A systolic loop on~$K$ is either a 1-sided
loop~$\gamma$, or a~$2$-sided loop~$\delta$ (see Lemma~\ref{ed}).

Now consider the real model \eqref{51} of~$\Sigma_2$, and its three
companion tori of type~\eqref{52}.  For each companion torus, we pass
to the quotient Klein bottle, and find a systolic loop satisfying
Bavard's bound.  We thus obtain three loops~$\delta_{a,b}$,
$\delta_{b,c}$, and~$\delta_{a,c}$.  If all three are two-sided, they
lift to loops~$\tilde\delta_{a,b}\subset\T_{a,b}$,
$\tilde\delta_{b,c}\subset\T_{b,c}$,
and~$\tilde\delta_{c,a}\subset\T_{c,a}$ on the tori.
Let~$\Delta_{a,b}\subset \hat\C$ be the projection of the loop
$\tilde\delta_{a,b}$ to the sphere, and similarly for~$\Delta_{b,c}$
and~$\Delta_{c,a}$.  Each loop~$\Delta\subset \hat\C$ defines a
partition of the 6-point set~$\Bra(\Sigma_2)\subset \hat\C$.

\begin{proposition}
Assuming the loops~$\delta_{a,b}$, etc., are~$2$-sided, build the
corresponding loops~$\Delta_{a,b}\subset \hat\C$, etc.  Consider the
three partitions of~$\Bra(\Sigma_2)\subset \hat\C$ defined by the
three loops~$\Delta\subset S^2$.  If the three partitions are not
identical, then there is a Bavard loop on~$\Sigma_2$.
\end{proposition}

\begin{proof}
Consider a Bavard systolic loop~$\delta \subset K$ of a Klein
bottle~$K$.  Consider its lift~$\tilde\delta\subset\T$ to the torus,
and the projection~$\Delta\subset \hat\C$.  If two such loops are
non-homotopic in
\[
S^2\setminus \{\Bra(\Sigma_2)\},
\]
we apply the cut and paste technique of \cite{KS1} to rearrange
segments of the two loops into a pair of loops that lift to closed
paths on the genus~$2$ surface.  One of the lifts is necessarily
Bavard.
\end{proof}

\begin{proposition}
\label{36}
If a systolic loop on a companion Klein bottle is~$1$-sided, then we
can transplant it to the Klein surface~$3\RP^2$, which therefore
satisfies Bavard's inequality.
\end{proposition}

\begin{proof}
Given a genus~$2$ surface \eqref{51}, consider a companion Klein
bottle~$K_{a,b} = \T_{a,b}/\tau$.  Consider a systolic loop
on~$K_{a,b}$.  If a systolic loop~$\gamma$ is one-sided, then~$\gamma$
lifts to a path connecting a pair of points in an orbit of~$\tau$ on
the torus.  The proof is completed by combining Lemma~\ref{path} and
Lemma~\ref{loop} below.
\end{proof}

\begin{lemma}
\label{path}
Let~$\gamma\subset K$ be a 1-sided loop, and let~$\tilde
\gamma\subset\T$ be the circle which is the connected double cover
of~$\gamma$.  Then there is a pair of real points~$p, \tau(p) \subset
\tilde\gamma$ which decompose~$\tilde\gamma$ into a pair of paths:
\[
\tilde\gamma= \gamma_+ \cup \gamma_- ,
\]
such that each of the paths~$\gamma_+$,~$\gamma_-$ projects to a
closed curve~$\Gamma_+$,~$\Gamma_-$ on~$\hat\C$.
\end{lemma}

\begin{proof}
Note that~$\tilde \gamma \subset\T$ is invariant under the fixed
point-free action of~$\tau$ on the torus.  The loop~$\tilde \gamma$
projects to a loop denoted
\[
\Gamma \subset \hat\C
\]
under the hyperelliptic quotient~$Q: \T\to \hat\C$.
Let~$\hat\R\subset\hat\C$ be the fixed point set of~$\tau$ acting
on~$\hat\C$.  The connected loop~$\Gamma\subset\hat\C$ is invariant
under complex conjugation.  Therefore it must meet the equator
$\hat\R$ in a point~$p_0\in\hat\R$ (see Section~\ref{six}).  Let
\[
Q^{-1}(p_0)= \{ p, \tau(p) \} \subset \T.
\]
Note that~$p_0$ is a self-intersection point of~$\Gamma$.  It may be
helpful to think of~$\Gamma$ as a figure-eight loop, with its
center-point~$p_0$ on the equator~$\hat\R$.  

We now view the original loop~$\gamma\subset K$ as a path starting at
the image of the point~$p$ in~$K$.  We lift the path~$\gamma$ to a
path~$\gamma_+\subset \T$ joining~$p$ and~$\tau(p)$.  The
path~$\gamma_+$ projects to half the loop~$\Gamma\subset \hat\C$,
forming one of the hoops of the figure-eight.  If we let~$\gamma_-=
\tau(\gamma_+)$, we can write~$\tilde\gamma= \gamma_+ \cup \gamma_-$.
Thus the loop~$\Gamma \subset \hat\C$ is the union of two loops
$\Gamma = Q(\gamma_+) \cup Q(\gamma_-)$.
\end{proof}

The following lemma refers to the five surfaces appearing in
diagram~\eqref{53}.

\begin{lemma}
\label{loop}
Let~$\gamma_+\subset \T_{a,b}$ be the path constructed in
Lemma~\ref{path}.  Consider the loop~$Q(\gamma_+)\subset \hat\C$, and
lift it to a path~$\tilde{\tilde\gamma} \subset \Sigma_2$.  Then the
image of~$\tilde{\tilde\gamma}$ under the covering
projection~$\Sigma_2 \to3\RP^2$ is a closed curve.
\end{lemma}

\begin{proof}
To transplant~$\gamma_+$ to the Klein surface~$3\RP^2$, note that the
path~$\tilde{\tilde\gamma} \subset \Sigma_2$ may or may not close up,
depending on the position of the third pair~$(c,\bar c)$ of branch
points of~$\Sigma_2\to \hat\C$.  If~$\tilde{\tilde\gamma}$ is already a
loop, then it projects to a Bavard loop on the Klein
surface~$3\RP^2=\Sigma/\tau$.  In the remaining case, the
path~$\tilde{\tilde\gamma}$ connects a pair of opposite points~$\tilde
p, \tau(\tilde p)$ on the surface~$\Sigma_2$, where~$Q(\tilde p)=
Q(\tau(\tilde p))= p_0$.  Therefore~$\tilde{\tilde\gamma}$ projects to
a Bavard loop in this case, as well.
\end{proof}

By Proposition~\ref{36}, it remains to consider the case when each
systolic loop~$\delta$ of each of the three companion Klein
bottles~$K_{a,b}, K_{a,c}, K_{b,c}$ is~$2$-sided.  Thus each of these
Bavard loops~$\delta_{a,b},\delta_{a,c},\delta_{b,c}$ lifts to a
closed curve~$\tilde\delta$ on the corresponding torus.  Let~$\Delta =
Q(\tilde\delta) \subset \hat\C=S^2$ be the corresponding loop on the
sphere.

\begin{lemma}
\label{meet}
If~$\Delta$ meets the equator, then the original Klein
surface~$3\RP^2$ contains a Bavard loop.
\end{lemma}

\begin{proof}
Let~$p_0 \in \Delta \cap \hat \R$.  Let~$p, \tau(p) \in \Sigma_2$ be
the points above it in~$\Sigma_2$.  We lift the path~$\Delta$ starting
at~$p_0$ to a path~$\delta_+ \subset \Sigma_2$ starting at~$p$.
If~$\delta_+$ closes up, its projection to~$3\RP^2$ is the desired
Bavard loop.  Otherwise, the path~$\delta_+$ connects~$p$
to~$\tau(p)$.  In this case as well, the path~$\delta_+$ projects to a
Bavard loop on~$3\RP^2=\Sigma_2/\tau$.
\end{proof}

It remains to consider the case when~$\Delta\cap\hat\R=\emptyset$.
This corresponds to the annulus decomposition case of
Proposition~\ref{141}, once we show that the loop is simple in the
following lemma.

\begin{lemma}
Let~$\tilde\delta$ be a systolic loop on the torus, and
let~$\Delta=Q(\tilde\delta) \subset \hat\C=S^2$ be the corresponding
loop on the sphere.  Assume that~$\Delta\cap\hat\R=\emptyset$.  Then
the loop~$\Delta$ is simple.
\end{lemma}

\begin{proof}
By hypothesis, the loop~$\Delta$ lies in a hemisphere, i.e., one of
the connected components of~$\hat\C\setminus\hat\R$.  The typical case
of a non-simple loop to keep in mind is a figure-eight curve.  Denote
by~$a,b\in\hat\C$ the branch points with respect to
which~$\Delta$ has odd winding number.  We will think of the
curve~$\Delta$ as defining a connected graph~$\AAA \subset \R^2$ in
a plane.  The vertices of the graph are the self-intersection points
of~$\Delta$.  Each vertex necessarily has valence~$4$.  By adding
the bounded ``faces'' to the graph, we obtain a ``fat'' graph
$\AAA_{{\rm fat}}$ (the typical example is the interior of the
figure-eight).  More precisely, the complement~$\C \setminus \AAA$ has
a unique {\em unbounded\/} connected component, denoted~$E \subset \C
\setminus \AAA$.  Its complement in the plane, denoted
\[
\AAA_{{\rm fat}}= \C \setminus E,
\]
contains both the graph~$\AAA$ and its bounded ``faces''.
Since~$\Delta$ has odd winding number with respect to the branch
points~$a,b$ of the double cover~$Q: \T_{a,b}^2 \to \hat\C$, they must
both lie inside the connected region~$\AAA_{{\rm fat}}$:
\[
a,b\in \AAA_{{\rm fat}}.
\]
The boundary of~$\AAA_{{\rm fat}}$ can be parametrized by a closed
curve~$\ell$, thought of the boundary of the outside
component~$E\subset \C$ so as to define an orientation on~$\ell$ (in
the case of the figure-eight loop, this results in reversing the
orientation on one of the hoops of the figure-eight).  Note that~$\ell
\subset \AAA$ is a subgraph.  Since both branch points lie inside, the
loop~$\ell$ had odd winding number with respect to each of the points
$a$ and~$b$.  Hence~$\ell$ lifts to a noncontractible loop on the
torus.  If~$\Delta$ is not simple, then the boundary of the outside
region~$E$ is not smooth, i.e., the loop~$\ell$ must contain
``corners'' and can therefore be shortened, contradicting the
hypothesis that~$\delta$ is a systolic loop.
\end{proof}

\section{Improving Gromov's~$3/4$ bound}
\label{fourteen}

In this section, we will prove the main theorem for $n=3$.

\begin{theorem}
\label{15b}
Let $\beta=\sqrt{1.333} \simeq 1.1545$.  The bound
\[
\sys \leq \beta \sqrt{\area}
\]
is satisfied by every metric on Dyck's surface~$3\RP^2$.
\end{theorem}

The orientable double cover of the hyperbolic Dyck's surface~$3\rp^2$
with the maximal systole was described by Parlier~\cite{Par}.
Silhol~\cite{Si, Si10} identified a presentation of its affine form:
\[
y^2=x^6+ ax^3+1, \quad a=434+108 \sqrt{17}.
\]
See also Leli\`evre \& Silhol \cite{LS} and Gendulphe~\cite{Ge}.  The
maximal systole on a hyperbolic Dyck's surface~$3\rp^2$ is equal to
$\arc\!\cosh \frac{5+\sqrt{17}}{2} = 2.19\ldots$ resulting in a
systolic ratio of $0.76\ldots$. \\

To prove Theorem~\ref{15b}, we use the
partition~$3\RP^2=\Sigma_{1,1}\cup \Mob$, as constructed in
Proposition~\ref{141}.  Since we are studying a scale-invariant
systolic ratio, we may and will normalize~$3\RP^2$ to unit area:
\[
\area(3\RP^2)=1.
\]
By Proposition~\ref{32}, we can assume that the metric on~$3\rp^2$
is~$J$-invariant.

By Proposition~\ref{141}, our desired bound on the systolic ratio of
the Klein surface~$3\RP^2=\Sigma_2/\tau$ reduces to the case when the
Bavard loops~$\delta$ on the companion Klein bottles are 2-sided, but
their lifts~$\tilde\delta$ to the triplet of tori~$\T$ project to
loops~$\Delta\subset\hat\C$, where each of the three loops~$\Delta$
defines the same partition of the set~$\Bra(\Sigma_2)\subset\hat\C$.
By Lemma~\ref{meet}, we may assume each~$\Delta$ lies in the open
northern hemisphere~$\hat\C^+$, and that its lift to~$\Sigma_2$
produces a non-closed curve.  Thus, we may assume that each of the
simple loops~$\Delta \subset \hat\C$ separates the six points~$a,\bar
a, b, \bar b, c, \bar c$ into two triplets~$(a,b,c)$ and~$(\bar a,
\bar b, \bar c)$.  Hence its connected double cover in~$\Sigma_2$ is
isotopic to the equatorial circle~$Q^{-1}(\hat\R)\subset\Sigma_2$.
The latter is a double cover of the equator~$\hat\R\subset\hat\C$ (see
Section~\ref{six}). \\

\forget
We consider two cases:
\begin{enumerate}
\item
$\area(\Sigma_{1,1})\leq 0.324$;
\item
$\area(\Sigma_{1,1})\geq 0.324$.
\end{enumerate}
We will find a short loop on~$\Sigma_{1,1}$ in the case (1), and on
the M\"obius band in the case (2).
\forgotten

We will start with a few preliminary topological results.

\begin{lemma} \label{lem:dist}
We have $\sys(3\RP^2)\leq 2 \, \dist(\Bra(3\RP^2), \text{Eq})$, where
$\text{Eq}$ is the equator of~$\C^+$.

In particular, if $\sys(3\rp^2) \geq \beta$ then 
\begin{equation}
\label{42}
\dist(\Bra(3\RP^2), \text{Eq}) \geq \tfrac{\beta}{2}.
\end{equation}
\end{lemma}

\begin{proof}
Every path on $\hat\C$ connecting a branch point to the
equator is double covered by a noncontractible loop in $3\RP^2$.  
The lemma follows.
\end{proof}

\forget
We will therefore assume throughout that
\begin{equation}
\label{42}
\dist(\Bra(3\RP^2), \text{Eq}) \geq \tfrac{\beta}{2}.
\end{equation}
We consider two cases:
\begin{enumerate}
\item
$\area(\Sigma_{1,1})\leq 0.324$;
\item
$\area(\Sigma_{1,1})\geq 0.324$.
\end{enumerate}
We will find a short loop on~$\Sigma_{1,1}$ in the case (1), and on
the M\"obius band in the case (2).

\section{Dyck's surfaces}

To prove Theorem~\ref{15b}, we need to locate an essential loop of
length at most~$\beta$, where
\[
1.154556 \leq \beta \leq 1.154557.
\]
If there is no such loop, the distance from each of the three branch
points to the equator must be at least~$\tfrac{1}{2}\beta$.  Note
that~$0.866133 \leq \beta^{-1} \leq 0.866134$.  We define~$h\in\R$ as
follows.

\begin{definition}
Let~$h= \tfrac{1}{2}(\beta-\beta^{-1}) \approx 0.144211$.
\end{definition}
\forgotten

\begin{definition}
Define $Y \subset \Sigma_2$ as the preimage of~$\hat{\C}^+$ under the
ramified cover $\Sigma_2 \to \hat{\C}$.  Observe that $Y$ is a
punctured torus and that $3\rp^2 = Y / \tau$ by identifying the
opposite points on~$\partial Y$ by~$\tau$.
\end{definition}

\begin{lemma} \label{lem:retrac}
If $sys(3\rp^2) \geq \beta$ then every arc of~$Y$ with endpoints
in~$\partial Y$ of length less than $0.6276$ is homotopically trivial
in~$\pi_1(Y,\partial Y)$.
\end{lemma}

\begin{proof}
Consider a length-minimizing arc~$c$ of~$Y$ with endpoints
in~$\partial Y$ which is not homotopic to an arc of~$\partial Y$
keeping its endpoints fixed.  The arc~$c$ is a nonselfintersecting
geodesic made of two minimizing segments of the same length meeting at
a point~$x$ with $\length(c) = 2 \, \dist(x,\partial Y)$.  If $c$ and
$Jc$ agree (up to orientation) then the arc~$c$ passes through a
Weierstrass point (which agrees with~$x$) and so $\length(c) \geq
\beta$ from~\eqref{42}.  We will therefore suppose otherwise.

\medskip

By Proposition~\ref{141}, the $J$-invariant simple loop $\delta
\subset 3\rp^2$ of length at most $2C_{Bavard}$ lifts to a
$J$-invariant simple loop in~$Y$ with the same length.  This loop will
still be denoted by~$\delta$.

By construction, the loops~$\delta$ and~$\partial Y$ bound a cylinder
in~$Y$.  In particular, the arc~$c$ intersects~$\delta$ at exactly two
points by minimality of~$c$ and~$\delta$, \cf~\cite{FHS}.  These two
points decompose~$c$ into three subarcs $c'$, $c_1$ and~$c_2$ with
$c'$ lying in~$\Sigma_{1,1}$, \cf~Figure~\ref{fig:abc}, that is,
$c=c_1 \cup c' \cup c_2$.  Switching $c_1$ and $c_2$ if necessary, we
can assume that $c_1$ is no longer than~$c_2$.

\begin{figure}[htbp]
\setlength\unitlength{1pt}
\begin{picture}(0,0)(0,0)
\put(-97,-96){\Large $a'$}
\put(-25,-96){\Large $c'$}
\put(34,-96){\Large $Jc'$}

\put(-41,-30){\Large $c_1$}
\put(-41,-158){\Large $c_2$}
\put(50,-158){\Large $Jc_1$}

\put(0,-161){\Large $b'$}

\put(99,-120){\Large $\delta$}
\put(130,-120){\Large $\partial Y$}

\put(-61,-10){\Large $c$}
\put(51,-10){\Large $Jc$}
\end{picture}

\begin{center}
\includegraphics[height=60mm]{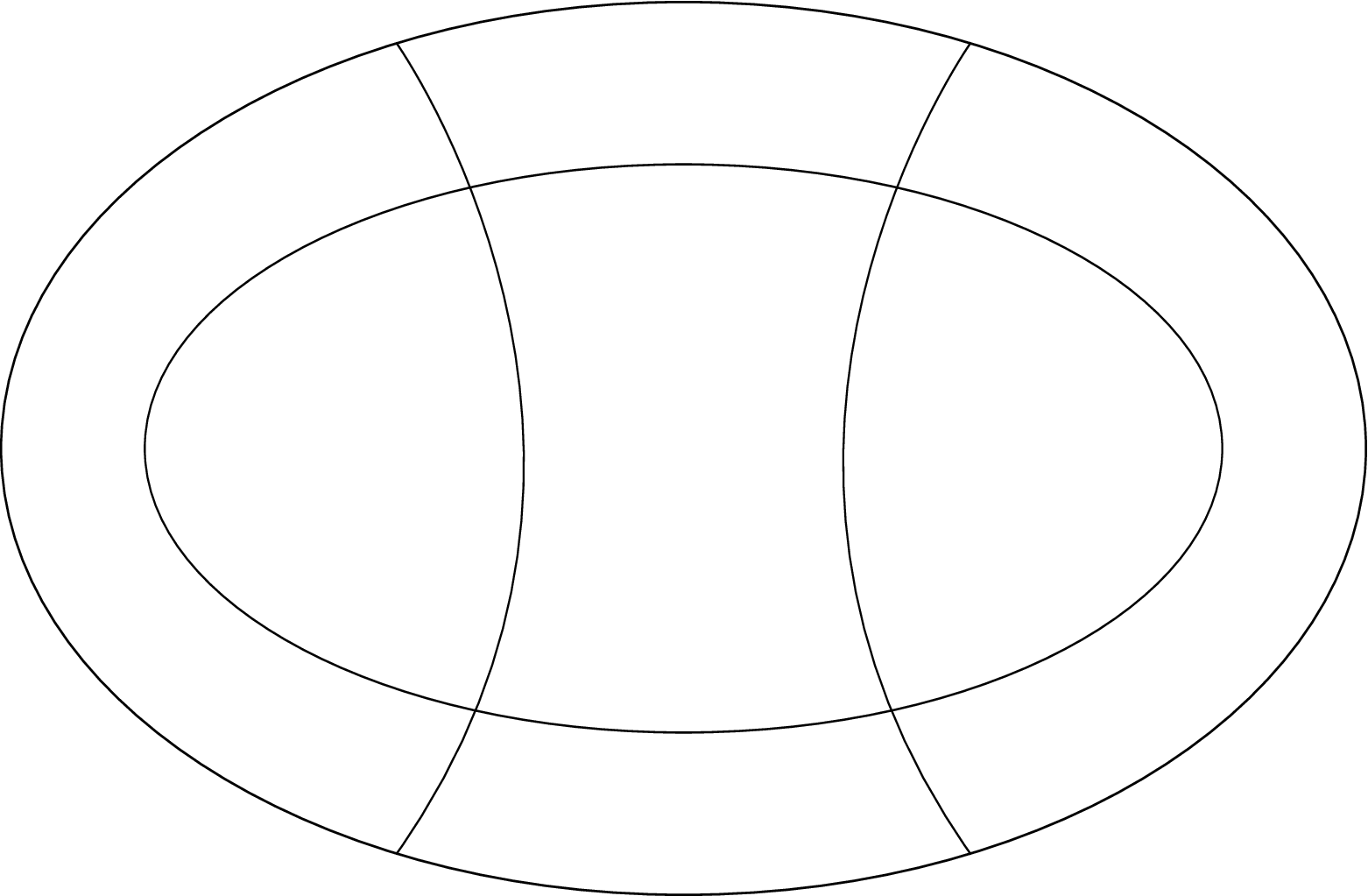}
\end{center}
\caption{Intersection properties of $c$ and $\delta$}
\label{fig:abc}
\end{figure}

The endpoints of~$c'$ and~$Jc'$, where $J$ is the hyperelliptic
involution on~$Y$, decompose~$\delta$ into four arcs $a'$, $Ja'$, $b'$
and~$Jb'$, see Figure~\ref{fig:abc}.  Since $c$ and~$\delta$ are
length-minimizing in their homotopy classes, the loop $a' \cup c'$ is
noncontractible in~$3\rp^2$ and so of length at least~$\beta$.  That
is,
\begin{equation} \label{eq:1}
\length(a') + \length(c') \geq \beta.
\end{equation}
By construction, the arc $c_1 \cup c' \cup b' \cup Jc_1$ with
symmetric endpoints induces a noncontractible loop in~$3\rp^2$.
Indeed, its $\Z_2$-intersection with the projection of~$\partial Y$
in~$3\rp^2$ is nontrivial.  Since the length of this arc is less or
equal to the sum of the lengths of~$c$ and~$b'$, we obtain
\begin{equation} \label{eq:2}
\length(c) + \length(b') \geq \beta.
\end{equation}
On the other hand, we have
\begin{eqnarray}
\length(a') + \length(b') & = & \frac{1}{2} \, \length(\delta)
\nonumber \\ & \leq & C_{Bavard} \label{eq:3}
\end{eqnarray}
Thus, from~\eqref{eq:1}, \eqref{eq:2} and~\eqref{eq:3}, we derive
\[
\length(c) \geq \frac{2\beta-C_{Bavard}}{2} \geq 0.6276.
\]
\end{proof}

We can now introduce the following definition.

\begin{definition}
Let $U_h$ be the $h$-neighborhood of the equator in $3\RP^2$ with 
\[
h=\frac{1}{2}(\beta - \beta^{-1}) \simeq 0.144211.
\]
From Lemma~\ref{lem:retrac} (under the assumption that $\sys(3\rp^2)
\geq \beta$), the boundary components of~$U_h$ are formed of one loop
freely homotopic to the equator and possibly several other
contractible loops.  Strictly speaking, we should first replace $h$
with a nearby regular value of the distance function from the equator.

Define now $\hat{U}_h$ as the union of~$U_h$ and the disks bounded by
its contractible boundary components.  Note that $\hat{U}_h$ is a
M\"obius band which does not contain any Weierstrass point from
Lemma~\ref{lem:dist}
\end{definition}

\begin{proposition}
\label{54}
If $\sys(3\rp^2) \geq \beta$ then $\area(\hat{U}_h)\geq 0.324$.
\end{proposition}

\begin{proof}
This area lower bound follows from an estimate due to
Blatter~\cite{Bl, Bl2} and used by Sakai~\cite{Sak}.  This is a lower bound for the
area of a M\"obius band in terms of its systole and length of path which
is non-homotopic to the boundary, and with endpoints at the boundary
circle.  The lower bound equals half the area of a spherical belt
formed by an~$h$-neighborhood of the equator of a suitable sphere of
constant curvature.  Here the equator of the suitable sphere has
length~$2\beta$ and its radius is $r = \frac{2\beta}{2\pi}= \frac{\beta}{\pi}$. 
The antipodal quotient of the spherical belt is a
M\"obius band of systole~$\beta$.  

Let~$\gamma$ be the
subtending angle~$\gamma$ of the northern half of the belt.
Then~$\gamma$ satisfies
\[
\gamma = \frac{h}{r} = \frac{h\pi}{\beta}=h\pi\beta^{-1}\geq 0.392403.
\]
(here we use the values~$h\geq 0.144211$ and~$\beta^{-1}\geq
0.866133$).  The height function is the moment map (Archimedes's
theorem), and hence the area of the belt is proportional to~$\sin
\gamma \geq 0.382434$.  The area of the corresponding region on the
unit sphere is~$4\pi \sin \gamma$.  Hence the area of the spherical
belt is~$4 \pi r^2 \sin \gamma = \frac{4 \pi \beta^2 \sin
  \gamma}{\pi^2}$, which after quotienting by the antipodal map yields
a lower bound
\[
\area(\hat{U}_h) \geq\frac{2\beta^2\sin\gamma}{\pi}= \frac{2 (1.333) \sin
\gamma}{\pi} \geq 0.324539,
\]
proving the proposition.
\end{proof}

Let us show that the same lower bound holds for the area of the M\"obius band~$\Mob$ in~$3\rp^2$.
For this purpose, we will need the following result.

\begin{lemma}
\label{53b}
If $\sys(3\RP^2) \geq \beta$ then $\length(\Delta) \leq \beta^{-1}$.
\end{lemma}

\begin{proof}
Let~$d_{eq}: \hat\C^+ \to \R$ be the distance function from the equator of~$\C^+$ endowed with the metric inherited from~$3\rp^2$.  
The equator is at distance at least~$\tfrac{1}{2}\beta$ from each of the three isolated branch points by Lemma~\ref{lem:dist}.
Each noncontractible connected component of the level curves of~$d_{eq}$ at distance at most $\frac{1}{2} \beta$ from the equator is longer than~$\Delta$. 
The coarea inequality gives a lower bound for the area of the strip~$S$ formed by the noncontractible connected components of the level curves corresponding to distances
\[
h \leq d_{eq} \leq \tfrac{1}{2}\beta.
\]
By construction, this strip is disjoint from the projection of~$\hat{U}_h$ on~$\C^+$.

If~$\length(\Delta) \geq \beta^{-1}$, we obtain the following area lower
bound for the strip:
\[
\area(S) \geq
\beta^{-1}\left(\tfrac{\beta}{2}-h\right)=\tfrac{1}{2}\beta^{-2}\geq
0.374.
\]
Now, the strip~$S$ lifts to a region on
$\Sigma_{1,1}$ of double the area, namely area at least~$0.748$, disjoint from~$\hat{U}_h$.
Combined with the lower bound of~$0.324$ for the area of $\hat{U}_h$ as in
Proposition~\ref{54}, this gives a total lower bound
\[
\area(3\RP^2) \geq  2 \, \area(S) + \area(\hat{U}_h) \geq 0.748 + 0.324 = 1.072
\]
for the area of~$3\RP^2$, which contradicts the original normalisation
$\area(3\RP^2)=1$.
\end{proof}

\begin{proposition}
\label{base}
If $\sys(3\RP^2) \geq \beta$ then the loop~$\Delta$ is at distance at least~$h$ from the equator.
In particular, $\area(\Mob) \geq 0.324$.
\end{proposition}

\begin{proof}
By contraposition, we consider an arc~$\gamma$ of length less than~$h$ connecting~$\Delta$ to the equator.
By Lemma~\ref{53b}, we have~$\length(\Delta) \leq \beta^{-1}$.  
Then the path
\[
\gamma \cup \Delta \cup \gamma^{-1}
\]
produces a noncontractible loop on $3\RP^2$ of length less than
\[
\length(\Delta) + 2h \leq \beta^{-1}+ 2h= \beta,
\]
proving the first statement of the proposition.

The second statement follows from Proposition~\ref{54} since the M\"obius band~$\Mob$ in~$3\RP^2$ contains~$\hat{U}_h$.
\end{proof}

\forget
\begin{proposition}
If $\area(\Mob)\leq 0.324$ then $\sys(3\RP^2)\leq\beta$.
\end{proposition}

\begin{proof}
If the loop $\Delta$ does not come within distance~$h$ of the equator,
then the M\"obius band $\Mob \subset 3\RP^2$ contains the
$h$-neighborhood of the equator, and therefore satisfies the area
lower bound of Proposition~\ref{54}, contradicting our hypothesis.  Hence
the loop $\Delta$ necessarily comes within distance~$h$ of the
equator, and the proposition results from Proposition~\ref{base}.
\end{proof}

The remaining case to consider is the case when $\area(\Mob)\geq
0.324$ and therefore~$\area(\Sigma_{1,1})\leq 0.676$.

\begin{proposition}
\label{163}
If~$\area(\Sigma_{1,1})\leq 0.676$,
then~$\sys(\Sigma_{1,1})\leq\beta$.
\end{proposition}
\forgotten

We can now proceed to the proof of Theorem~\ref{15b}.

\begin{proof}[Proof of Theorem~\ref{15b}]
Suppose that $\sys(3\rp^2) \geq \beta$.
The surface $3\RP^2$ is separated into the union
\[
3\RP^2 = \Sigma_{1,1} \cup \Mob,
\]
where $\area(\Sigma_{1,1})\leq 0.676$ and $\area(\Mob)\geq 0.324$ from Proposition~\ref{base}.
Here, the separating loop is isometric to a circle twice as long as~$\Delta$, and therefore of radius
\begin{equation*}
r = \frac{1}{\pi} \, \length(\Delta).
\end{equation*}
The area of a hemisphere based on such a circle is
\begin{equation*}
2\pi r^2 = \frac{2}{\pi}\, \length(\Delta)^2 \leq \frac{2\beta^{-2}}{\pi}
\end{equation*}
by Lemma~\ref{53b}.  Attaching the hemisphere to the torus with a disk
removed produces a torus of total area at
most~$\frac{2}{\pi\beta^2}+0.676$.  Applying Loewner's
bound~\eqref{11e} to the resulting torus, we obtain a systolic loop of
square-length at most
\begin{equation*}
\frac{2}{\sqrt{3}}\left(\frac{2}{\pi\beta^2}+0.676\right)\leq 1.333,
\end{equation*}
proving the theorem.
\end{proof}

\section{Other hyperelliptic surfaces} \label{sec:last}

In this section, we prove Theorem~\ref{theo:main} in the remaining
case~$n \geq 4$.  Recall that a non-orientable surface is called {\em
hyperelliptic\/} if its orientable double cover is.

\begin{proposition}
Let~$n\geq 4$.  Every Riemannian metric from a hyperelliptic conformal
type on the surface~$n\rp^2$ satisfies the bound
\[
\frac{\sys^2}{\area} \leq \left( \frac{1}{4} +
\frac{n}{8} \right)^{-1}.
\]
In particular,
\[
\sys^2 \leq 1.333 \; \area.
\]
\end{proposition}

\begin{proof}
Let $L=\sys(X)$.  Without loss of generality, we can assume that the
hyperelliptic invariant metric on~$X$ has the property that the area
of every disk $B(R)$ of radius~$R$ with $R \leq L/2$ satisfies
\begin{equation}
\label{61} 
\area(B(R)) \geq 2R^2,
\end{equation}
see \cite[Lemma~3.5]{KS1}.

For $n$ even with $n=2k \geq 4$, the orientable double cover of the
surface~$X=\Sigma_{k-1} \# K$ is a hyperelliptic
surface~$\Sigma_{2k-1}$ of genus $2k-1$.  As in
Lemma~\ref{lem:circle}, the preimage of the equatorial
circle~$\hat{\R} \subset \hat{\C}$ under the double cover
$\Sigma_{2k-1} \to \hat{\C}$ is a pair of disjoint circles.  This pair
of circles bounds the preimage $Y \subset \Sigma_{2k-1}$ of the
northern hemisphere~$\hat{\C}^+$ under the previous double cover.  The
orientation-reversing involution~$\tau$ on~$\Sigma_{2k-1}$ switches
the two boundary components of~$Y$.  We can obtain~$X$ from~$Y$ by
identifying the pairs of points of~$\partial Y$ corresponding to the
orbits of the involution~$\tau$.  Alternatively, we can define~$Y$ by
compactifying the open surface $X \setminus \pi({\rm Fix}(J \circ
\tau))$, where $\pi: \Sigma_{2k-1} \to X$ is the quotient map induced
by~$\tau$.

Let $\gamma$ be a length-minimizing arc of~$Y$ joining the two
boundary components of~$Y$.  The hyperelliptic involution~$J$ takes
the endpoints of the $1$-chain $\gamma \cup (-J\gamma)$ of~$Y$ to
their opposite.  As these endpoints lie in~$\partial Y$, and since $J$
and~$\tau$ agree on~$\partial Y$, the previous $1$-chain of~$Y$
induces a loop~$c$ in~$X$.  This loop is noncontractible in~$X$.
Indeed, let $X'$ be the complex with fundamental group isomorphic
to~$\Z$ obtained by collapsing the region of~$X$ outside a
sufficiently small tubular neighborhood of the equator to a point.  By
construction, the loop~$c$ of~$X$ projects to a loop of~$X'$
representing twice a generator of~$\pi_1 X' \simeq \Z$.  Thus, $c$ is
noncontractible in~$X$ and so of length at least~$L$.  We deduce that
the distance between the two boundary components of~$Y$ is at
least~$\frac{L}{2}$.  Similarly, the distance between the Weierstrass
points of~$Y$ and its boundary components is at least~$\frac{L}{2}$,
see Lemma~\ref{lem:dist}.  Likewise, the distance between any pair of
Weierstrass points is at least~$\frac{L}{2}$.  This shows that the
open disks~$D_i$ of radius~$\frac{L}{4}$ centered at the $2k$
Weierstrass points of~$Y$ and the open $\frac{L}{4}$-tubular
neighborhoods, $U_1$ and~$U_2$, of the boundary components of~$Y$ are
pairwise disjoint.

Now, as in Lemma~\ref{lem:dist}, the level curves at distance at
most~$\frac{L}{4}$ from the boundary components of~$Y$ project to
curves that separate the isolated branch points from the equator
in~$\hat{\C}^+$.  Thus, each of these level curves has a
noncontractible component in~$X$ and so is of length at least~$L$.
From the coarea inequality, we deduce that the area of each tubular
neighborhood~$U_j$ is at least~$\frac{L^2}{4}$.  Since the metric
satisfies~\eqref{61}, the area of~$D_i$ is at least~$\frac{L^2}{8}$.
Adding these lower bounds, we obtain
$$
\area(X) \geq \left( \frac{1}{2} + \frac{n}{8} \right) \, L^2.
$$

\medskip

For $n$ odd with $n=2k+1 \geq 5$, the orientable double cover
of~$X=\Sigma_{k} \# \rp^2$ is a hyperelliptic surface~$\Sigma_{2k}$ of
genus $2k$.  As previously, the preimage of the equilatorial circle
$\hat{\R} \subset \hat{\C}$ under the double cover $\Sigma_{2k} \to
\hat{\C}$ is a single circle.  This circle bounds the preimage~$Y$ of
the northern hemisphere~$\hat{\C}^+$ under the previous double cover.
The orientation-reversing involution~$\tau$ of~$\Sigma_{2k}$ takes the
points of~$\partial Y$ to their opposite points.  We can obtain~$X$
from~$Y$ by identifying the pairs of opposite points of~$\partial Y$.

As previously, the distances between the Weierstrass points of~$Y$
and~$\partial Y$, and between any pair of Weierstrass points are at
least~$\frac{L}{2}$.

Arguing as in the previous case, we deduce from the coarea inequality
that the area of the $\frac{L}{4}$-tubular neighborhood of~$\partial
Y$ is at least $\frac{L^2}{4}$.  Combined with the estimates on the
areas of the disjoint disks of radius~$\frac{L}{4}$ centered at the
Weierstrass points of~$Y$, we obtain
\[
\area(X) \geq \left( \frac{1}{4} + \frac{n}{8} \right) \, L^2,
\]
proving the proposition.
\end{proof}

\section*{Acknowledgments}

We are grateful to C. Croke for helpful discussions.


\end{document}